\documentclass[12pt]{amsart}
\usepackage{amsmath,amssymb,amsbsy,amsfonts,amsthm,latexsym,amsopn,amstext,
                                                    amsxtra,euscript,amscd,color}
\usepackage[english]{babel}
\begin{document}

\newtheorem{thm}{Theorem}
\newtheorem{lem}[thm]{Lemma}
\newtheorem{claim}[thm]{Claim}
\newtheorem{cor}[thm]{Corollary}
\newtheorem{prop}[thm]{Proposition} 
\newtheorem{definition}{Definition}
\newtheorem{question}[thm]{Open Question}
\newtheorem{conj}[thm]{Conjecture}
\newtheorem{prob}{Problem}
\def\vol {{\mathrm{vol\,}}}
\def\squareforqed{\hbox{\rlap{$\sqcap$}$\sqcup$}}
\def\qed{\ifmmode\squareforqed\else{\unskip\nobreak\hfil
\penalty50\hskip1em\null\nobreak\hfil\squareforqed
\parfillskip=0pt\finalhyphendemerits=0\endgraf}\fi}

\def\cA{{\mathcal A}}
\def\cB{{\mathcal B}}
\def\cC{{\mathcal C}}
\def\cD{{\mathcal D}}
\def\cE{{\mathcal E}}
\def\cF{{\mathcal F}}
\def\cG{{\mathcal G}}
\def\cH{{\mathcal H}}
\def\cI{{\mathcal I}}
\def\cJ{{\mathcal J}}
\def\cK{{\mathcal K}}
\def\cL{{\mathcal L}}
\def\cM{{\mathcal M}}
\def\cN{{\mathcal N}}
\def\cO{{\mathcal O}}
\def\cP{{\mathcal P}}
\def\cQ{{\mathcal Q}}
\def\cR{{\mathcal R}}
\def\cS{{\mathcal S}}
\def\cT{{\mathcal T}}
\def\cU{{\mathcal U}}
\def\cV{{\mathcal V}}
\def\cW{{\mathcal W}}
\def\cX{{\mathcal X}}
\def\cY{{\mathcal Y}}
\def\cZ{{\mathcal Z}}

\def\NmQR{N(m;Q,R)}
\def\VmQR{\cV(m;Q,R)}

\def\Xm{\cX_m}

\def \A {{\mathbb A}}
\def \B {{\mathbb A}}
\def \C {{\mathbb C}}
\def \F {{\mathbb F}}
\def \G {{\mathbb G}}
\def \L {{\mathbb L}}
\def \K {{\mathbb K}}
\def \Q {{\mathbb Q}}
\def \R {{\mathbb R}}
\def \Z {{\mathbb Z}}
\def \fS{\mathfrak S}

\def\\{\cr}
\def\({\left(}
\def\){\right)}
\def\fl#1{\left\lfloor#1\right\rfloor}
\def\rf#1{\left\lceil#1\right\rceil}

\def\Tr{{\mathrm{Tr}}}
\def\Im{{\mathrm{Im}}}

\def \bFp {\overline \F_p}

\newcommand{\pfrac}[2]{{\left(\frac{#1}{#2}\right)}}

\def \Prob{{\mathrm {}}}
\def\e{\mathbf{e}}
\def\ep{{\mathbf{\,e}}_p}
\def\epp{{\mathbf{\,e}}_{p^2}}
\def\em{{\mathbf{\,e}}_m}

\def\Res{\mathrm{Res}}
\def\Orb{\mathrm{Orb}}

\def\vec#1{\mathbf{#1}}
\def\flp#1{{\left\langle#1\right\rangle}_p}

\def\mand{\qquad\mbox{and}\qquad}

\newcommand{\commI}[1]{\marginpar{%
    \vskip-\baselineskip 
    \raggedright\footnotesize
    \itshape\hrule\smallskip\begin{color}{red}#1\end{color}\par\smallskip\hrule}}

\newcommand{\commO}[1]{\marginpar{%
    \vskip-\baselineskip 
    \raggedright\footnotesize
    \itshape\hrule\smallskip\begin{color}{blue}#1\end{color}\par\smallskip\hrule}}


\title[Polynomial Values in Subfields and 
Affine Subspaces]{Polynomial Values in Subfields and 
Affine Subspaces of Finite Fields}

\author[O. Roche-Newton]{Oliver Roche-Newton} 
\address{Johann Radon Institute for Computational and Applied Mathematics, Austrian Academy of Sciences, 4040 Linz, Austria}
\email{o.rochenewton@gmail.com}

\author[I. E. Shparlinski]{Igor E. Shparlinski} 
\address{School of Mathematics and Statistics, University of New South Wales, 
Sydney, NSW 2052, Australia}
\email{igor.shparlinski@unsw.edu.au}

\date{\today}

\begin{abstract} For an integer $r$, a prime power  $q$, 
and a polynomial $f$ over a finite field 
$\F_{q^r}$ of $q^r$ elements,  we obtain an 
upper bound on the frequency of elements in 
an orbit generated by iterations of $f$ which fall in 
a proper subfield of $\F_{q^r}$. We also obtain similar 
results for elements in affine subspaces of  $\F_{q^r}$, 
considered as a linear space over $\F_q$.
\end{abstract}

\subjclass[2010]{11T06, 37P05, 37P55}

\keywords{finite fields, polynomial dynamics, orbits}

\maketitle

\section{Introduction}

 \subsection{Background}

For a prime power  $q$ and an integer $r>1$ we
consider finite fields $\K = \F_q$ and $\F = \F_{q^r}$ of 
$q$ and $q^r$ elements, respectively. 

The motivation behind this work comes from some questions
of polynomial dynamics, however to address these questions 
we first obtain new results which fall in the domain of 
additive combinatorics as are of independent interest.

More precisely,  given a polynomial $f\in \F[X]$ and an 
element $u \in \F$, we define the orbit
\begin{equation}
\label{eq:Orbu}
\Orb_{f} (u)= \{f^{(n)}(u)\ : \ n = 0, 1,\ldots\},
\end{equation}
where $f^{(n)}$ is the $n$th iterate of $f$, that is,
$$
f^{(0)}=X,\quad f^{(n)}=f(f^{(n-1)}),\quad n\ge 1.
$$ 
Here we consider the question about the frequency 
of elements in orbits $\Orb_{f} (u)$ that fall in the 
proper subfield $\K \subseteq \F$. Our first result is 
based on some combinatorial argument and shows that unless 
some iterate $f^{(s)}$ of $f$ is defined over $\F$
(for a rather small $s$) then the frequency of this event 
is low.

Furthermore, we also study the frequency of orbit elements that fall in 
an affine subspace of $\F$ considered as a linear vector space over $\K$.
This question is motivated by a recent work of 
Silverman and Viray~\cite{SiVi} (in characteristic zero 
and using a very different technique), see 
also~\cite{AKNTVV}. Using new results from additive combinatorics 
we obtain a lower bound on the 
dimension of an affine space 
that may contain $N$ consecutive elements in an orbit. This result 
may also be considered as an analogue of the results on 
the diameters of polynomial orbits in prime fields $\F_p$,
see~\cite{Chang1,Chang2,CCGHSZ,CGOS,GuShp}. More precisely, 
in~\cite{Chang1,Chang2,CCGHSZ,CGOS,GuShp} various lower bounds are 
given on the length $H$ of the shortest interval $[a+1, a+H]$ 
that contains residues modulo $p$ of the $N$ consecutive iterations 
$f^{(n)}(u)$, $n =0, \ldots, N-1$.


There are also related results where bounds on the size of the 
intersection $\#\(f(\cA) \cap \cB\)$ are given, 
where $\cA$ and $\cB$ are some `interesting' sets and
$f(\cA) = \{f(a)~:~a \in \cA\}$ is the value set 
of a polynomial $f$ on $\cA$; see~\cite{CCGHSZ,CGOS} for the 
case when both sets $\cA$ and $\cB$ are intervals of consecutive
integers and~\cite{GomShp,Shp2}  for the case when $\cA$ is such 
an interval and $\cB$ is a multiplicative subgroup of $\F_p$. 
Unfortunately in the very interesting case when both sets 
$\cA$ and $\cB$ are subgroups of $\F_q$ no results are known. 
We also note that bounds on $\#\(f(\cA) \cap \cB\)$ for intervals 
$\cA$ and $\cB$ play an important role in the analysis of some
algorithms~\cite{G-MRST}. 

Finally, we also mention that the intersection 
of $\cL^{-1} \cap \cM$ for two linear subspaces of $\F$ over $\K$ 
has been studied by Mattarei~\cite{Matt1,Matt2} by using different methods
(certainly $0$ should be discarded from $\cL$ in the definition 
of $\cL^{-1}$).
It is quite likely that the methods of additive combinatorics 
can be applied to this question as well. 

 \subsection{Our results}
 \label{sec:results}
 
Here we consider affine 
subspaces in high degree extensions of finite fields as natural 
analogues of intervals. To obtain results about orbits in 
affine subspaces of $\F$,  we extend  
a result of Bukh and Tsimerman~\cite{BuTs} on a polynomial 
version of the celebrated sum-product theorem
to the case of arbitrary finite fields. 

Recall that 
 Bukh and Tsimerman~\cite[Theorem~1]{BuTs} give a lower bound on 
$ \max \{\#(\cA+\cA), \#\(f(\cA)+f(\cA)\)\}$ for subsets $\cA$ of 
prime fields. Their technique can be extended to arbitrary fields.
However, motivated by our application we obtain a result of this
kind with multifold sums of the set $\cA$, see Theorem~\ref{thm:BT-arbF} 
below, which in turn gives stronger versions of our principal results. 
We believe that this result and some technical innovations 
in its proof can be 
of independent interest as well and may have several other applications. 
For example, using this result 
we derive an upper bound on the intersection 
 $\#\(\A \cap f(\A)\)$
of an affine subspace $\A$ of $\F$  over $\K$
and its polynomial image $f(\A)$, see Theorem~\ref{thm:PolySubspace}
below.

Furthermore,  let us define the dimension $\dim \cS$ of a set $\cS \subseteq \F$
as the smallest dimension of all affine subspaces of $\F$  over $\K$
that contain $\cS$. 
In Corollary~\ref{cor:PolyDim} we obtain 
a lower bound on $\dim f(\A)$ for an affine subspace $\A$ of $\F$  over $\K$.
 Questions 
of this type sometimes appear in theoretical computer science, see~\cite{B-SG,B-SK}
and references therein. 
Note that some results of this kind for very special affine subspaces are also
given in~\cite{CillShp}. 

Finally, as we have mentioned, we apply these results to achieve our 
main goal: bound on the frequency of polynomial orbits in subspaces.

\subsection{Notation}
Throughout the paper, any implied constants in the symbols $O$,
$\ll$  and $\gg$ may depend on $\deg f$. 
We recall that the
notations $U = O(V)$, $U \ll V$ and  $V \gg U$ are all equivalent to the
statement that the inequality $|U| \le c V$ holds with some
constant $c> 0$.

\section{Polynomial Version of Sum-Product Estimates}

\subsection{Preparations}

We now obtain a version of the result of Bukh and  Tsimerman~\cite[Theorem~1]{BuTs}
for polynomials over arbitrary finite fields.  

As usual given $m$ sets $\cA_1, \ldots, \cA_m
\subseteq \F$ and a polynomial
$$F(X_1, \ldots, X_m) \in
\F[X_1, \ldots, X_m], 
$$ 
we define the set
$$
F(\cA_1, \ldots, \cA_m) =\{F(a_1, \ldots, a_m) \ : \ (a_1, \ldots,  a_m)
\in \(\cA_1 \times \ldots \times \cA_m\)\}. 
$$

In particular, $\cA + \cA$ and $\cA \cdot \cA$ are the sum set and 
the product set of $\cA$, respectively.  

We note that in our version of~\cite[Theorem~1]{BuTs} we use the 
set  $\cA+\cA+\cA+\cA-\cA-\cA-\cA-\cA$ instead of $\cA+\cA$ and $f(\cA)-f(\cA)$ instead of 
$f(\cA)+f(\cA)$, which leads for stronger expansion factor and is more suitable 
for our applications. 

Throughout, the  notation $\cA:\cB$ is used to denote 
the ratio set (assuming that $\cB \subseteq \F^*$).
Furthermore, we need the idea of  a {\it restricted ratio set\/}.
Namely, if  $\cE\subseteq \cA\times \cB$, then the ratio set of $\cA$ and $\cB$ restricted to $\cE$ is the set
$$
\cA\stackrel{\cE}{:}\cB=\{a/b~:~(a,b)\in \cE\}.
$$

The following result is a small modification of~\cite[Theorem~1.4]{LiR-N}. The 
necessary sum-ratio estimate is mentioned (without a proof) in~\cite{LiR-N}; 
a full proof can be found in~\cite{R-N}.

\begin{lem} \label{lem:SP Gen} Let $\cA$ be a subset of $\F$ with the property that for any subfield $\G$, and any $a \in \F$,
$$ \# (\cA\cap a\G)\leq \max{\left\{(\#\G)^{1/2},\frac{\#\cA}{8}\right\}}.$$
Then either,
$$\(\#(\cA:\cA)\)^4\(\#(\cA+\cA+\cA+\cA)\)^5\gg (\#\cA)^{10}
$$
or
$$\(\#(\cA:\cA)\)^5\(\#(\cA+\cA+\cA+\cA)\)^4\gg (\#\cA)^{10}.
$$
\end{lem}

We also define 
$$
\cA\stackrel{\cE}-\cB=\{a-b~:~(a,b)\in \cE\}.
$$
Furthermore, we extend in a natural way the definition of 
the sum set $\cA + \cB$ and difference set $\cA  - \cB$
to subsets of an arbitrary group, written additively.


We also require  a version 
of the Balog-Szemer\'{e}di-Gowers Theorem. 
The following result is essentially given in~\cite[Lemma~2.2]{BouGar}. Note that in the statement of \cite[Lemma~2.2]{BouGar} it is assumed that the group in question is the additive group of the prime field $\F_p$, however, one can verify that the same proof works for an arbitrary group $\cG$.

\begin{lem} 
\label{lem:BSG}
Let $\cG$ be an arbitrary group, written additively, and let $\cU,\cV\subseteq \cG$. 
Let $\cE\subseteq {\cU\times{\cV}}$ such that 
$$\#\cE\geq{\frac{\#\cU\#\cV}{K}}.
$$
for some real $K\ge 1$. Then there exists a subset $\cU_0\subseteq\cU$ 
such that 
$$\#\cU_0\geq{\frac{\#\cU}{10K}} \mand 
\#\(\cU\stackrel{\cE}{-}\cV\)^4\geq{\frac{\#(\cU_0-\cU_0)\#\cU(\#\cV)^2}{10^4K^5}}.$$
\end{lem}

Finally, we use a form of the Pl\"{u}nnecke-Ruzsa inequality 
which follows from~\cite[Theorem~1.1.1]{Ruzsa} (see also~\cite[Lemma~9]{BuTs}).

\begin{lem} 
\label{lem:PR}
Let $\cG$ be an arbitrary group, written additively, and let $\cU\subseteq \cG$. 
Then 
$$
\#\(\cU+\cU -\cU - \cU\)^4\le \(\frac{\#(\cU-\cU)}{\#\cU}\)^4 \#\cU.$$
\end{lem}

\subsection{Main result}

Let $p$ be the characteristic of  $\F = \F_{q^r}$.

\begin{thm}
\label{thm:BT-arbF}
Let $\cA$ be a subset of $\F$ with the property that, for any subfield $\G \subseteq \F$ and 
any $a\in \F$,
$$\#\((\cA-\cA)\cap a \G \)\leq
 \max{\left\{(\#\G)^{1/2},\frac{\#\cA^{1-\vartheta_d}}{8}\right\}} .$$
Then, for any polynomial $f\in \F [X]$ of degree $d = \deg f$ 
with $p> d \ge 2$ we have
\begin{equation*}
\begin{split}
\max \{\#(\cA+\cA+\cA+\cA-\cA-\cA-\cA-\cA),&\#\(f(\cA)-f(\cA)\)\}\\
& \qquad  \geq c_d(\# \cA)^{1+\eta_d},
\end{split}
\end{equation*}
where  $\eta_2=\vartheta_2 =  1/69$, and  
$c_2 =  c$ for some absolute constant $c>0$
and then 
$$
\eta_d = \frac{\eta_{d-1}}{5+\eta_{d-1}}, \quad 
\vartheta_d = \vartheta_{d-1} +  \eta_d - \vartheta_{d-1} \eta_d,
\quad c_d= \(\frac{c_{d-1}}{d^3}\)^{1/(5+\eta_{d-1})}
$$ 
for $d \ge 3$. 
\end{thm}

\begin{proof} Let $\# \cA = M$. We define $\alpha, \beta, \gamma$ and $\xi$ 
by the relations: 
$$\#
(\cA+\cA)=\alpha M, \quad \#(\cA - \cA)=\beta M, \quad \#\(f(\cA)-f(\cA)\)= \xi M.
$$
and
$$\#(\cA+\cA+\cA+\cA-\cA-\cA-\cA-\cA)=\gamma M
$$ 

The proof uses induction on $d$. 

Consider first the base case $d=2$.
In this case the condition on the set $\cA$ is simply 
\begin{equation}
\label{eq:cond d=2}
\#\((\cA-\cA)\cap a \G \)\leq
 \max{\left\{(\#\G)^{1/2},\frac{(\#\cA)^{1-\eta_2}}{8}\right\}} .
\end{equation}

 Without loss of generality, we can assume that 
the polynomial $f$ is monic and with the zero constant coefficient, since the cardinality of $f(\cA)-f(\cA)$ does not vary under these changes to $f$. 
The polynomial $f$ can then be written as $f(x)=x^2+bx$, for some $b \in \F$. Note that, 
for any $x,y\in \cA$,
\begin{equation}
\label{eq:fact}
f(x)-f(y)=x^2+bx - y^2 - by =  (x-y)(x+y+b). 
\end{equation}

Next, define a set $\cE \subseteq(\cA-\cA)\times{(\cA+\cA+b)^{-1}}$ by the equation
$$
\cE=\{(x-y,(x+y+b)^{-1})~:~x,y\in \cA ,x+y+b\neq 0\}.$$

Note that for $p> 2$ each pair $(x,y)$ with $x,y\in \cA$ and  $x+y+b\neq 0$ 
leads to a different element of $\cE$. 
Hence 
\begin{equation}
\label{eq:size E} 
\# \cE \geq{M^2-M}\geq{\frac{M^2}{2}
 =\frac{\#\(\cA-\cA\)\# \(\cA+\cA+b\)}{2\alpha \beta}}.
\end{equation}

Moreover, by~\eqref{eq:fact} we have
\begin{equation}
\label{eq:include}
(\cA-\cA)\stackrel{\cE}{:}(\cA+\cA+b)^{-1}\subseteq f(\cA)-f(\cA),
\end{equation}
where we define
$$
(\cA+\cA+b)^{-1} = \{(x+y+b)^{-1}~:~x,y\in \cA ,x+y+b\neq 0\}.
$$ 

We now apply Lemma~\ref{lem:BSG} in this setting, with
with the group $\cG = \F^*$, the sets
$$\cU=(\cA-\cA) \backslash \{0\} \mand \cV=(\cA+\cA+b)^{-1}, 
$$
and thus,
by~\eqref{eq:size E}, with $K = 2 \alpha \beta$. 
Without loss of generality we can assume that $\# \cA \ge 2$, 
so $\#\(\cA-\cA\)  \ge \# \cA \ge 2$, and thus
$$
\#\(\cA-\cA\)-1 \ge \frac{1}{2} \#\(\cA-\cA\).
$$
Hence there exists a subset $\cA_0\subseteq {\cA-\cA}$ such that
\begin{equation}
\label{eq:A0 large}
\# \cA_0\geq\frac{\#\(\cA-\cA\)-1}{20\alpha \beta}\ge 
\frac{\#\(\cA-\cA\)}{40\alpha \beta} =
\frac{M}{40\alpha}
\end{equation}
and
\begin{equation*}
\begin{split}
\#\((\cA-\cA)\stackrel{\cE}{:}(\cA+\cA+b)^{-1}\)^4
& \geq \frac{\#\(\cA_0:\cA_0\)\alpha^2 \beta M^3}
{10^4 (2 \alpha  \beta)^5}\\
& =  \frac{\#\(\cA_0:\cA_0\)  M^3}
{32\cdot 10^4  \alpha^3  \beta^4}.
\end{split}
\end{equation*}

Applying the upper bound on the restricted ratio set which comes 
from~\eqref{eq:include}, and simplifying, gives
$$\alpha^3  \beta^4 \xi^4 M \geq\frac{\#\(\cA_0:\cA_0\) } {32\cdot 10^4}.
$$

If $\alpha > M^{\eta_2}/40$ 
there is nothing to prove. 
Otherwise we see from~\eqref{eq:A0 large} that $\# \cA_0\geq M^{1-\eta_2}$.
Note that it now follows from~\eqref{eq:cond d=2} that 
\begin{equation*}
\begin{split}
\#\( \cA_0\cap a \G \)& \leq
\#\((\cA-\cA)\cap a \G \)\\&
\leq
 \max{\left\{(\#\G)^{1/2},\frac{M^{1-\eta_2}}{8}\right\}}
\le 
 \max{\left\{(\#\G)^{1/2},\frac{\# \cA_0}{8}\right\}}   .
\end{split}
\end{equation*}
Therefore Lemma~\ref{lem:SP Gen} applies to $\cA_0$ and if interpreted as a lower bound for 
$\#\(\cA_0:\cA_0\)$ yields either
\begin{equation}
\(\#\(\cA_0+\cA_0+\cA_0+\cA_0\)\)^5 \alpha^{12}\beta^{16}\xi ^{16}M^4
\gg{\(\#\cA_0\)^{10}}\gg{\frac{M^{10}}{\alpha^{10}}},
\label{option1}
\end{equation}
or
\begin{equation}
\(\#\(\cA_0+\cA_0+\cA_0+\cA_0\)\)^4 \alpha^{15}\beta^{20}\xi ^{20}M^5
\gg{\(\#\cA_0\)^{10}}\gg{\frac{M^{10}}{\alpha^{10}}}.
\label{option2}
\end{equation}

Note that 
$$\cA_0+\cA_0+\cA_0+\cA_0\subseteq \cA+\cA+\cA+\cA-\cA-\cA-\cA-\cA,
$$ 
so that 
$\#(\cA_0+\cA_0+\cA_0+\cA_0)\leq \gamma M$. It is also straightforward to check that 
$\alpha, \beta \leq \gamma$. Putting this information into \eqref{option1}, we conclude that
\begin{equation}
\gamma^{43}\xi^{16}\gg M .
\label{case1}
\end{equation}
On the other hand, the  inequality~\eqref{option2} gives
\begin{equation}
\gamma^{49}\xi^{20}\gg{M}.
\label{case2}
\end{equation}
Since~\eqref{case1} also implies~\eqref{case2} it can be concluded that
\begin{equation*}
\begin{split}
\max \{\#(\cA+\cA+\cA+\cA-\cA-\cA-\cA-\cA),&\#\(f(\cA)-f(\cA)\)\}\\
& \qquad  \gg (\# \cA)^{1+\frac{1}{69}}. 
\end{split}
\end{equation*}
Taking a sufficiently small value of $c$ we obtain the 
desired result for $d=2$, which concludes the base case. 

Now assume that the result holds with $d-1$ instead of $d$. 

Let 
$$
r(t) = \# \{(x,y) \in \cA \times \cA~:~ t = x-y\}.
$$
Since 
$$
\sum_{t \in \cA-\cA} r(t)=M^2,
$$ 
it follows that there exists some $t \in \cA-\cA$ such that 
$$
r(t)\geq{\frac{M^2}{\#(\cA-\cA)}}=\frac{M}{\beta}.
$$
Define $\cB=\{a\in \cA~:~a+t\in \cA\}$, and so 
\begin{equation}
\label{eq:B large}
\# \cB\ge M/\beta.
\end{equation}

Now, if $\beta > M^{\eta_{d}}$ 
then there is nothing to prove. 
Otherwise 
we have 
\begin{equation}
\label{eq:B very large}
\#\cB \ge M/\beta \ge M^{1-\eta_d}.
\end{equation}
 
We now define a new polynomial $g(X)=f(X+t)-f(X)$, and note that $\deg g =d-1$ as $\deg f < p$. 
It is easy to check that $\cB$ satisfies the subfield intersection conditions. Indeed,
using~\eqref{eq:B very large} we derive
\begin{equation*}
\begin{split}
\#\((\cB-\cB)\cap a \G \)& \leq \#\((\cA-\cA)\cap a \G \) \le 
 \max{\left\{(\#\G)^{1/2},\frac{M^{1-\vartheta_d}}{8}\right\}} \\
& \le  \max{\left\{(\#\G)^{1/2},\frac{(\#\cB)^{(1-\vartheta_{d})/(1-\eta_d)}}{8} \right\}}\\
& =  \max{\left\{(\#\G)^{1/2},\frac{(\#\cB)^{1-\vartheta_{d-1}}}{8} \right\}}.
\end{split}
\end{equation*}
So the 
inductive hypothesis can be applied with $g$ and $\cB$. 
There are two possibilities; either
\begin{equation}
\label{eq:choice1}
\#\(g(\cB)-g(\cB)\) \ge c_{d-1} \(\# \cB\)^{1+\eta_{d-1}}
\end{equation}
or 
\begin{equation}
\label{eq:choice2}
\#\(\cB+\cB+\cB+\cB-\cB-\cB-\cB-\cB\) \ge c_{d-1} \(\# \cB\)^{1+\eta_{d-1}}.
\end{equation}

If~\eqref{eq:choice1} holds, we note that 
$g(\cB)-g(\cB) \subseteq f(\cA)+f(\cA)-f(\cA)-f(\cA)$. 
Also, by Lemma~\ref{lem:PR}, we have 
$$
\#\(f(\cA)+f(\cA)-f(\cA)-f(\cA)\)
\leq{\frac{\(\#\(f(\cA)-f(\cA)\)\)^4}{\(\#f(\cA)\)^3}}.
$$
So, using the trivial bound $\#f(\cA) \ge M/d$, we obtain
$$
c_{d-1} \(\# \cB\)^{1+\eta_{d-1}} \leq d^3 \xi ^4 M.
$$
Recalling~\eqref{eq:B large} we derive 
$$
c_{d-1}d^{-3} M^{\eta_{d-1}} \le \beta^{1+\eta_{d-1}}\xi^4.
$$
Therefore 
$$
\max\{\beta, \xi\} \ge \(c_{d-1}d^{-3} M^{\eta_{d-1}}\)^{1/(5+\eta_{d-1})}, 
$$
and so the desired result holds.

If~\eqref{eq:choice2} holds, then it follows from the fact that $\cB\subseteq \cA$ that
$$
c_{d-1}  \(\# \cB\)^{1+\eta_{d-1}} \leq  \gamma  M\leq{\gamma Md^3}.
$$
By applying~\eqref{eq:B large}, we derive 
$$
c_{d-1}d^{-3} M^{\eta_{d-1}} \le \beta^{1+\eta_{d-1}}\gamma. 
$$
Therefore, using the fact that $\gamma \ge \beta$ we obtain 
$$
\max\{\gamma, \xi\} \ge \gamma \ge \(c_{d-1}d^{-3} M^{\eta_{d-1}}\)^{1/(2+\eta_{d-1})}, 
$$
and so in this case  the desired result  holds as well.  This closes the induction
and concludes the proof. 
\end{proof}

It is easy to see that 
$$
\lim_{d\to \infty} \frac{\log \eta_d}{d} = - \log 5.
$$
Note that a similar result can also be obtained for the exponent of the
bound of  Bukh and  Tsimerman~\cite[Theorem~1]{BuTs} (instead of 
$16^{-1}\cdot 6^{-d}$). This in turn 
implies
that 
$$
\lim_{d\to \infty} \frac{\log c_d}{\log d} = - 3. 
$$

We now recall the definition of the dimension $\dim \cS$ of a set $\cS \subseteq \F$
as the smallest dimension of all affine subspaces of $\F$  over $\K$
that contain $\cS$.

\begin{cor} \label{cor:PolyDim} Let $f\in \F[X]$ be of degree $d=\deg f$ 
with $p> d \ge 2$. Let  $\A \subseteq \F$
be an affine subspace of dimension $s$ over $\K$ 
such that for any subfield $\G \subseteq \F$ and any $a\in \F$,  we have
$$
\#\(\cL \cap a\G  \) \leq  \max{\left\{(\#\G)^{1/2},\frac{q^{s(1-\vartheta_d)}}{8}\right\}}, 
$$ 
where $\A = b + \cL$ for some $b\in \F$ and a linear subspace $\cL \subseteq \F$. 
Then 
$$
\dim f(\A) \ge (1+\eta_d + o(1)) \dim \A, 
$$
as $\# \A \to \infty$ where  $\eta_d$ and $\vartheta_d$ are as in
Theorem~\ref{thm:BT-arbF}. \end{cor}

\begin{proof} Since we obviously have 
$\#(\A+\A+\A+\A-\A-\A-\A-\A) = \#\A$  and $\A - \A = \cL$, then   
Theorem~\ref{thm:BT-arbF} implies $\#\(f(\A)-f(\A)\)\} \geq c_d(\# \A)^{1+\eta_d}$
and the result follows. 
\end{proof}

\subsection{Some remarks on Theorem~\ref{thm:BT-arbF}}

Regarding the condition in Theorem~\ref{thm:BT-arbF}  that the degree of $f$ satisfies $d \leq p$, we note that some
condition is necessary in order to account for the possibility that our polynomial $f$ is additive. If $f$ has the property that $f(x+y)=f(x)+f(y)$ for all $x,y\in{\mathbb{L}}$, then we can take $\cA$ to be 
an affine subspace of $\F$ 
and observe that $\# (\cA+\cA), \#(f(\cA)+f(\cA)) \ll \# \cA$. For example, if $f(X)=X^p$, then $f$ is an additive polynomial, and this is why the inductive argument breaks down at this point.

The condition that $\cA-\cA$ does not have an overly large intersection with a dilate of a subfield is needed in order to apply the sum-product estimate from \cite{LiR-N}. Again, some condition of this kind is necessary, since it could be the case that $\cA=\G$ for some subfield $\G\subseteq \F$. Then, if the coefficients of $f$ are all taken from $\G$, we obviously have 
$$\cA+\cA+\cA+\cA-\cA-\cA-\cA-\cA, \ f(\cA)-f(\cA)\subseteq \cA,$$
and so the estimate in Theorem \ref{thm:BT-arbF} does not hold. It seems likely that the result holds under the cleaner condition that $\cA$ does not have an overly large intersection with any subfield. We 
note that if $\# \cA \geq  (\#\F)^{1/2}$, then this simplification of the condition can be obtained, since one does not need to worry about the subfield intersection conditions in the sum-product estimate for larger subsets of a finite field. Sum-product estimates for large subsets of a finite field can be found in ~\cite{Gar} and~\cite{Vinh}.

It remains an interesting and open problem to give a full classification of the polynomials $f$ and sets $\cA$  for which 
$$
\max \{\#(\cA+\cA),\#\(f(\cA)-f(\cA)\)\} = (\# \cA)^{1+o(1)}, 
$$
as $\# \cA \to \infty$. 

\section{Distribution of Polynomial Orbits}

\subsection{Polynomials orbits in subfields} 

Clearly, for any $u \in \F$,  the orbit~\eqref{eq:Orbu}
is a finite set as the sequence $f^{(n)}(u)$, $n = 0,1,\ldots$,
is eventually periodic.  Let $T_u = \# \Orb_{f} (u)$ 
be the size of the orbit. 

We now show that if a segment of an orbit of length $N \le T_u$ 
has a large intersection with $\K$ then there is an
iterate of $f$ which is defined over $\K$.

We note that the argument of this section works for any fields
$\K \subseteq \F$, not necessary finite fields.  

\begin{thm} \label{thm:Subfield} Let $f\in \F[X]$ be of degree $d \ge 1$ and let
 $u \in \F$. Assume that for some real   $\eta >0$
and an integer $N \le T_u$ 
we have 
$$
\#  \{0 \le n < N~:~f^{(n)}(u)\in \K\} \ge c(d)\frac{N}{\log N} +1 , 
$$
where $c(d) = 2 \log (4d)$. 
Then for some integer $k$ we have 
$f^{(k)}(X)\in \K[X]$. 
\end{thm}

\begin{proof} 
Let  $1\leq{n_1}< \ldots < n_M\leq{N}$ be  all values with
 the property that $f^{(n_i)}(u)\in \K$.
We denote by $A(h)$ the number of $i=1, \ldots, M-1$ 
with $n_{i+1} - n_i = h$. 
Clearly 
$$
\sum_{h=1}^{N} A(h) = M-1 \mand
\sum_{h=1}^{N} A(h)h = n_{M}- n_1 \le N.
$$
Thus for any integer $H \ge 1$ we have 
\begin{equation*}
\begin{split}
\sum_{h=1}^{H} & A(h) = M-1 - \sum_{h=H+1}^N A(h) \\
&\ge   M-1 - (H+1)^{-1}\sum_{h=H+1}^N A(h)h 
\ge  M-1 - (H+1)^{-1}N.
\end{split}
\end{equation*}

Hence there exists $k \in \{1, \ldots, H\}$ 
with 
\begin{equation}
\label{eq:A(h) H}
A(k) \ge H^{-1} \(M-1  -  (H+1)^{-1}N\). 
\end{equation}
We now set 
$$
H =\fl{\frac{2N}{(M-1)}}. 
$$ 
Clearly $H \ge 1$.
Then 
$$
H^{-1} \(M-1  -  (H+1)^{-1}N\)\ge \frac{M-1}{2H} \ge \frac{N}{H(H+1)} 
$$
and we derive from~\eqref{eq:A(h) H} that 
\begin{equation}
\label{eq:A(h)}
A(k) \ge \frac{N}{H(H+1)} .
\end{equation}
Assume that $d^k \ge A(k)$. Then the inequality~\eqref{eq:A(h)}
implies that 
$$
d^H \ge d^k\ge  \frac{N}{H(H+1)} .
$$
Since $H <H+1 \le 2^H$ for $H\ge 1$, we derive
$$
(4d)^H > H(H+1)d^H  \ge N
$$
which in turn implies that 
$$
\frac{2N}{(M-1)} \geq H  > \frac{\log N}{\log (4d)}
$$
which contradicts our assumption on the frequency 
of orbit elements that belong to $\K$.

Therefore,  $d^k < A(k)$. 

Let $\cJ$ be the set of $j \in \{0, \ldots, M-1\}$ 
with $n_{j+1} - n_j = k$. 
Then we have 
$$
f^{(n_j)}(u)\in \K 
\mand f^{(n_{j+1})}(u) = f^{(k)}\(f^{(n_j)}(u)\) \in \K . 
$$
Since 
$$
\deg f^{(k)} = d^k <  A(k) = \# \cJ
$$
we now see that $f^{(k)}(w) \in \K$ for more than $\deg f^{(k)}$ 
elements $w \in \K$. 
Then by Lagrange interpolation we have
$f^{(k)}(X)\in \K[X]$, which concludes the proof.
\end{proof}

\subsection{Polynomials orbits in affine subspaces} 

As before we  denote by $p$ the characteristic of $\F = \F_{q^r}$.

\begin{thm} \label{thm:PolySubspace} Let $f\in \F[X]$ be of degree $d=\deg f$ 
with $p> d \ge 2$ and let $\A \subseteq \F$ be an  affine subspace of dimension $s$ over $\K$
such that for any subfield $\G \subseteq \F$ and any $a\in \F$ we have
$$
\#\(\cL \cap a\G  \) \leq  \max{\left\{(\#\G)^{1/2},\frac{q^{s(1-\rho_d)}}{8}\right\}}, 
$$ 
where $\A = b + \cL$ for some $b\in \F$ and a linear subspace $\cL \subseteq \F$. 
Then 
$$
\#\(\A \cap f(\A)\) \ll  q^{s(1-\kappa_d)}, 
$$
where 
$$\kappa_d = \frac{\eta_d}{1+\eta_d} \mand \rho_d = \eta_d+\vartheta_d -  \eta_d\vartheta_d, 
$$
and  $\eta_d$ and $\vartheta_d$ are   as in Theorem~\ref{thm:BT-arbF}.
\end{thm}

\begin{proof}
Let $\cS = \A \cap f(\A)$. Now, for each $s \in \cS$ we choose an element $a\in \A$ with $f(a) = s$. 
Let $\cA$ be this set,  
so that $\# \cS = \# \cA$.

It is obvious that   
$$
\cA+\cA+\cA+\cA-\cA-\cA-\cA-\cA \subseteq \cL 
$$
and also
$$
f(\cA) - f(\cA) \subseteq \cS - \cS \subseteq \cL .
$$
If $\# \cA <  q^{s(1-\eta_d)}$ there is nothing to prove.
Otherwise 
$$
\#\cA^{1-\vartheta_d} \ge  q^{s(1-\eta_d)(1-\vartheta_d)} = \#\cL^{1-\rho_d}
$$
Hence,   since $\cA - \cA \subseteq \cL$  we have
\begin{equation*}
\begin{split}
\#\((\cA - \cA)\cap  a\G ) \)& \leq
\#\(\cL \cap  a\G  \) \le \max{\left\{(\#\G)^{1/2},\frac{\#\cL^{1-\rho_d}}{8}\right\}}\\
& \le \max{\left\{(\#\G)^{1/2},\frac{\#\cA^{1-\vartheta_d}}{8}\right\}}
\end{split}
\end{equation*}
for any $a\in \F$. Therefore
Theorem~\ref{thm:BT-arbF} applies to the set $\cA$ and implies that 
$$
\# \cL \gg \(\# \cA\)^{1+\eta_d}
$$
from which  we  immediately derive the result. 
\end{proof}

\begin{cor} \label{cor:OrbSubspace} Let $f\in \F[X]$ be of degree $d=\deg f$ 
with $p> d \ge 2$. Let  $\A \subseteq \F$
be an affine subspace of dimension $s$ over $\K$ 
such that for any subfield $\G \subseteq \F$ and any $a\in \F$  we have 
$$
\#\(\cL \cap a\G  \) \leq  \max{\left\{(\#\G)^{1/2},\frac{q^{s(1-\rho_d)}}{8}\right\}}, 
$$ 
where $\A = b + \cL$ for some $b\in \F$ and a linear subspace $\cL \subseteq \F$. 
If for some  $u \in \F$ and an integer $N$ with 
$2 \le N \le T_u$ 
we have 
$$
f^{(n)}(u)\in \A, \qquad n =0, \ldots, N-1, 
$$
then 
$$
q^s \gg N^{1+\eta_d},
$$
where $\eta_d$ and $\rho_d$ are as in 
Theorems~\ref{thm:BT-arbF} and~\ref{thm:PolySubspace},
respectively.
\end{cor}

\begin{proof} Let $\cR = \{f^{(n)}(u)~:~n =1, \ldots, N-1\}$. 
Then clearly,  under the condition of the theorem, we have 
$\cR \subseteq  \A \cap f(\A)$. Using Theorem~\ref{thm:PolySubspace} we derive the result. 
\end{proof}

\section{Comments}

It is certainly interesting to obtain a multiplicative 
analogue of Theorem~\ref{thm:BT-arbF} for the sets
$\cA\cdot \cA$ and $f(\cA) \cdot f(\cA)$ (and their multifold
analogues).  A result of this type  can be used to study the 
distribution of polynomial orbits in subgroups. We note that even 
over prime fields this question is still widely open, see~\cite{Shp1}.
It is related to the aforementioned open problem of estimating 
the size of the intersection $f\(\cG\)\cap \cH$ for two multiplicative 
subgroups $\cG, \cH \subseteq \F^*$. The case when $\cG = \cH$ is of 
direct relevance to studying orbits of dynamical systems in subgroups. 
It seems plausible that the 
method of Heath-Brown and Konyagin~\cite{HBK}, that has recently been
advanced by Shkredov~\cite{Shkr1,Shkr2}, is  able to yield  such results
over prime fields.

Studying rational functions instead of polynomials is
an interesting direction as well. 

The methods of proofs of Theorems~\ref{thm:Subfield} and~\ref{thm:PolySubspace}  
do not seem to 
extend to multivariate polynomials and it is very desirable to find
an alternative approach.

\section*{Acknowledgements}

The authors are grateful to Boris Bukh for helpful discussions and to Mike Zieve for information about 
his work in progress on polynomial orbits in subfields over fields of 
characteristic zero (however both the technique used 
and actually the results seem to be of a different nature).  

During the preparation of this paper
Oliver Roche-Newton was supported by the EPSRC Doctoral Prize Scheme,
Grant~EP/K503125/1 and part of this research was performed while he was visiting the Institute for Pure and Applied Mathematics (IPAM), which is supported by the National Science Foundation; Igor Shparlinski was supported by the   
Australian Research Council, 
Grants~DP130100237 and~DP140100118.


\begin{thebibliography}{9}

\bibitem{AKNTVV} E. Amerik, P. Kurlberg, K. Nguyen, A. Towsley, B. Viray
and J.~F.~Voloch, `Evidence for the dynamical Brauer-Manin criterion',
{\it Preprint\/}, 2013 (available from 
{\tt http://arxiv.org/abs/1305.4398}).

\bibitem{B-SG} E. Ben-Sasson and A. Gabizon,  `Extractors for polynomial 
sources over fields of constant order and small characteristic',
{\it Theory Comput.\/}, {\bf  9} (2013), 665--683. 

\bibitem{B-SK} E. Ben-Sasson and  S. Kopparty, `Affine dispersers from subspace polynomials',
{\it Proc.  41st Annual ACM Symp. Theory of Comp.\/}, ACM,  2009,   65--74.

\bibitem{BouGar} J. Bourgain and M. Z. Garaev,
`On a variant of sum-product estimates and explicit exponential sum bounds in prime fields',
{\it Math. Proc. Cambridge Philos. Soc.\/}, {\bf  146} (2009),  1--21.


\bibitem{BuTs} B. Bukh and J. Tsimerman,
`Sum-product estimates for rational functions', {\it
Proc. Lond. Math. Soc.\/}, {\bf 104}  (2012),  1--26. 

\bibitem{Chang1} M.-C. Chang, `Polynomial iteration in characteristic $p$',
{\it J. Functional Analysis\/},  {\bf  263} (2012), 3412--3421.

\bibitem{Chang2} M.-C. Chang, `Expansions of quadratic maps in prime fields',
{\it Proc. Amer. Math. Soc.\/},  {\bf  142} (2014), 85--92.

\bibitem{CCGHSZ} M.-C. Chang, J. Cilleruelo, M. Z. Garaev, J.  Hern\'andez,
I. E. Shparlinski and A. Zumalac\'{a}rregui,
`Points on curves in small boxes and applications',
{\it Michigan Math. J.\/},  (to appear).

\bibitem{G-MRST} O. Garcia-Morchon, R. Rietman, I. E. Shparlinski
and L. Tolhuizen, 
`Interpolation and approximation of polynomials in finite fields
over a short interval from noisy  values',  
{\it Experimental Math.\/}, (to appear).

\bibitem{CGOS} J. Cilleruelo, M. Z. Garaev,  A. Ostafe and
I. E. Shparlinski,
`On the concentration of points of polynomial maps
and applications',
{\it Math. Zeit.\/}, {\bf 272} (2012), 825--837.

\bibitem{CillShp} J. Cilleruelo and I. E. Shparlinski,
`Concentration of points on curves in finite fields',  
{\it  Monatsh. Math.\/},  {\bf 171} (2013), 315--327. 

\bibitem{Gar}  M. Z. Garaev,
`The sum-product estimate for large subsets of prime fields',
{\it  Proc. Amer. Math. Soc.\/}, {\bf  136} (2008),  2735--2739.

\bibitem{GomShp} D. G\'omez-P\'erez and I. E. Shparlinski,
`Subgroups generated by  rational 
functions in finite fields',  
{\it Preprint\/}, 2013  (available from {\tt
http://arxiv.org/abs/1309.7378}).

\bibitem{GuShp} J. Gutierrez and I.~E.~Shparlinski,
`Expansion of orbits of some dynamical systems over finite fields',
{\it Bull. Aust. Math. Soc.\/}, {\bf  82} (2010),  232--239.

\bibitem{HBK}
D. R. Heath-Brown and S. V. Konyagin, `New bounds for Gauss sums
derived from $k$th powers,
and for Heilbronn's exponential sum',
{\it Quart. J. Math.\/}, {\bf 51} (2000), 221--235.

\bibitem{LiR-N}  L. Li and O. Roche-Newton, 
`An improved sum-product estimate for general finite fields', 
{\it SIAM J. Discr. Math.\/}, {\bf 25} (2011),  1285--1296.

\bibitem{Matt1}  S. Mattarei, 
`Inverse-closed additive subgroups of fields', 
{\it Isr. J. Math.\/}, {\bf 159} (2007), 343--347. 

\bibitem{Matt2}  S. Mattarei, 
`A property of the inverse of a subspace of a finite field', 
{\it Finite Fields and Their Appl.\/}, {\bf 29} (2014), 268--274.

\bibitem{R-N}  O. Roche-Newton, 
`Sum-ratio estimates over arbitrary finite fields',
{\it Preprint\/}, 2014 (available from {\tt
http://arxiv.org/abs/1407.1654}).

\bibitem{Ruzsa} I. Z. Ruzsa, `Sumsets and structure', 
{\it Combinatorial Number Theory and
Additive Group Theory\/}, Birkh{\"a}user, 2009, 87--210.

\bibitem{Shkr1}  I. D. Shkredov, `Some new inequalities in 
additive combinatorics',
{\it Moscow J. Comb. and Number Theory\/},  (to appear).
 
\bibitem{Shkr2}  I. D. Shkredov, 
`On exponential sums over multiplicative subgroups of medium
size', {\it Preprint\/},  2013
 (available from {\tt http://arxiv.org/abs/1311.5726}).

\bibitem{Shp1} I. E. Shparlinski, `Groups generated by iterations 
of polynomials over finite fields', 
{\it Proc. Edinburgh Math. Soc.\/}, (to appear).

 \bibitem{Shp2} I. E. Shparlinski,
`Polynomial values in small subgroups of finite fields',  
{\it Preprint\/}, 2014  (available from {\tt
http://arxiv.org/abs/1401.0964}).

\bibitem{SiVi}  J. H. Silverman and B. Viray, 
`On a uniform bound for the number of exceptional linear 
subvarieties in the dynamical Mordell-Lang conjecture', 
{\it Math. Res. Letters.\/}, {\bf  20} (2013),  547--566. 

\bibitem{Vinh} L. A. Vinh, 
`The Szem{\'e}redi--Trotter type theorem and the sum-product estimate in finite fields', 
{\it European J. Combin.\/}, {\bf 32} (2011), 1177--1181. 

\end{thebibliography}
\end{document}